\documentclass[12pt, reqno]{amsart}
\usepackage{amsmath}

\usepackage{amsfonts}
\usepackage{amssymb}
\usepackage{mathrsfs}
\usepackage{comment}
\usepackage{graphicx}
\usepackage{xcolor}
\usepackage{enumitem}
\usepackage{physics}

\usepackage[margin=1in,footskip=0.25in]{geometry}
\newcommand{\di}{\,\mathrm{d}}
\providecommand{\U}[1]{\protect\rule{.1in}{.1in}}
\newtheorem{theorem}{Theorem}

\newtheorem{definition}[theorem]{Definition}

\newtheorem{lemma}[theorem]{Lemma}

\newtheorem{proposition}[theorem]{Proposition}
\newtheorem{remark}[theorem]{Remark}

\begin{document}

%\linenumbers

\title{Integral Operators on Fractional Ces\`aro--Morrey Spaces over $\mathbb{Q}_p$}

\author[Q. Jahan]{Qaiser Jahan}
\address[Q. Jahan]{School of Mathematical and Statistical Sciences,  Indian Institute of Technology Mandi, Kamand (H.P.), 175005,  India}
\email{qaiser@iitmandi.ac.in}

\author[R. Saini]{Rishabh Saini}

\address[R. Saini]{School of Mathematical and Statistical Sciences,  Indian Institute of Technology Mandi, Kamand (H.P.), 175005,  India}
\email{d23030@students.iitmandi.ac.in}

 \keywords{$p$-adic field, integral operator, Ces\`aro function spaces, Morrey space .}
\subjclass[2010]{11F85, 47G10}
\begin{abstract}
In this paper, we introduce fractional Ces\`aro--Morrey function spaces over $p$-adic fields and investigate their fundamental properties. We first study the behavior of dilation operators on these spaces and establish a Minkowski-type integral inequality. We then prove the boundedness of $p$-adic Hardy--Hilbert--type integral operators on fractional Ces\`aro--Morrey function spaces. As an applications, we derive the $p$-adic Hardy inequality, the Hilbert inequality, and the Hardy--Littlewood--P\'{o}lya inequality. Finally, we establish the boundedness of the Erd\'elyi--Kober  fractional integral operator and the Hadamard fractional integral operator on these spaces.
\end{abstract}
%%%%%%%%%%%%%%%%%%%%%%%%%%%%%%%%%%%%%%%%%%%%%%%%%%%%%%%%%%%%%%%%%%%%%
\maketitle

\section{Introduction}

The Ces\`aro sequence spaces arise from the classical Ces\`aro mean, introduced by Ernesto Ces\`aro in the nineteenth century as a method for summability of divergent series. For $1\le p<\infty$, the Ces\`aro sequence space $ces_p$ consists of all sequences $x=(x_k)_{k\ge1}$ such that
\[
\|x\|_{ces_p}
=
\left(
\sum_{n=1}^{\infty}
\left(
\frac{1}{n}\sum_{k=1}^{n}|x_k|
\right)^p
\right)^{1/p}
<\infty,
\]
and when $p=\infty$
\[
\|x\|_{ces_\infty}
=
\sup_{n\in\mathbb{N}}
\frac{1}{n}\sum_{k=1}^{n}|x_k|
<\infty .
\]

Motivated by the sequence setting, the corresponding function space analogue was introduced. Let $I=[0,1]$ or $I=[0,\infty)$. The Ces\`aro function space $Ces_p(I)$ consists of all Lebesgue measurable functions $f$ on $I$ satisfying
\[
\|f\|_{Ces_p(I)}
=
\left(
\int_I
\left(
\frac{1}{x}\int_0^x |f(t)|\,dt
\right)^p
dx
\right)^{1/p}
<\infty,
\qquad 1\le p<\infty,
\]
and
\[
\|f\|_{Ces_\infty(I)}
=
\sup_{x\in I,\;x>0}
\frac{1}{x}\int_0^x |f(t)|\,dt
<\infty .
\]

The Ces\`aro function spaces $Ces_r(I)$, defined on the intervals $I=[0,1]$ or $I=[0,\infty)$, form an important class of Banach function spaces that arise naturally in connection with averaging operators and integral inequalities. These spaces may be viewed as continuous counterparts of Ces\`aro sequence spaces. The study of Ces\`aro function spaces gained significant attention after a problem posed by the Dutch
Mathematical Society in 1968 on the characterization of their dual spaces \cite{AS1}. Since then, many of their fundamental properties have been investigated. In particular, Ces\`aro function spaces are, in general, neither rearrangement invariant nor reflexive, and their geometric structure differs substantially from that of the classical Lebesgue spaces. These properties make them a useful setting for the study of averaging operators, integral inequalities, and related
problems in analysis. For further results on the duality theory, geometric properties, and other aspects of Ces\`aro function spaces, we refer the reader to
\cite{AI,AS1,AS2,AS3,AS4,AS5,KT,KL2,KL3}.

%Closely related to Ces\`aro spaces are the Copson spaces and the Tandori spaces, which arise from similar averaging constructions and share several analytical properties with Ces\`aro spaces. These spaces have been investigated in connection with operator theory, interpolation, and embedding problems.

Another important class of function spaces is the Morrey spaces, introduced by Morrey in his study of the regularity of solutions to quasilinear elliptic partial differential equations. Since then, Morrey spaces have become a fundamental tool in harmonic analysis and partial differential equations, particularly in the study of the local behavior of functions. Over the years, several generalizations of Morrey spaces have been introduced, including generalized Morrey spaces, Orlicz--Morrey spaces, Morrey--Lorentz spaces, and
Morrey--Banach spaces \cite{GS2,GS1,GS3,NE1,NE2,SY}.

%Motivated by these developments, it is natural to consider Morrey type structures associated with Ces\`aro spaces. In particular, Ces\`aro Morrey spaces combine the averaging behavior of Ces\`aro operators with the local control inherent in Morrey spaces and provide a refined framework for studying integral operators and related inequalities.

The combination of Ces\`aro and Morrey spaces leads naturally to the Ces\`aro--Morrey spaces, which have been studied in the Euclidean setting. These spaces combine the averaging structure of Ces\`aro spaces with the local control provided by Morrey spaces, making them well suited for the study of integral operators and related inequalities see \cite{HO2, RSJ}.

A central problem in the theory of Ces\`aro-type spaces concerns the boundedness of integral operators acting on these spaces. In the classical setting on $[0,\infty)$ the Ces\`aro norm is closely connected with Hardy type inequalities, which has motivated the development of general criteria for operator boundedness. In this direction, Ho \cite{HO2} developed a unified approach based on the behavior of dilation operators together with Minkowski-type inequalities. This method provides a convenient framework for establishing the boundedness of a broad class of integral operators, including the Hilbert operator, the Erdélyi--Kober fractional integral operator, and Mellin-type fractional operators.

Although the boundedness theory of integral operators on Ces\`aro--Morrey spaces has been extensively developed in the classical Euclidean setting, to the best of our knowledge, the theory of fractional Ces\`aro--Morrey spaces over the $p$-adic field $\mathbb{Q}_p$ has not yet been developed. Even though the definition of fractional Ces\`aro--Morrey spaces over $\mathbb{Q}_p$ is motivated by the classical setting, the corresponding analysis is fundamentally different. This is due to the non-Archimedean nature of the $p$-adic absolute value, which induces an ultrametric geometry that differs
significantly from the Euclidean setting. In particular, the
decomposition of $\mathbb{Q}_p$ into $p$-adic spheres replaces
continuous radial integration by discrete summation formulas.
Consequently, the boundedness of integral operators is characterized by discrete summability conditions on the kernel rather than integral conditions, and many arguments from the Euclidean setting cannot be applied directly.

A related study of Ces\`aro-type function spaces over $\mathbb Q_p$ was recently initiated in \cite{RSJ}, where $p$-adic Ces\`aro function spaces and the boundedness of certain integral operators on these spaces were investigated. The present work is motivated by this development but is carried out in a broader functional framework. By introducing a fractional parameter together with a Morrey-type structure, we obtain a new class of fractional Ces\`aro--Morrey spaces over $\mathbb Q_p$. The $p$-adic Ces\`aro spaces considered in the earlier work arise naturally as a special case of the present framework. This broader setting enables us to investigate the boundedness of integral operators on spaces that combine fractional averaging with the Morrey-space structure.

In this paper we introduce and study fractional Ces\`aro--Morrey spaces over the $p$-adic field $\mathbb{Q}_p$. We first investigate the associated dilation operators and establish a Minkowski-type integral inequality, which plays a key role in the proofs of our main results. We then prove a general boundedness theorem for a class of $p$-adic Hardy--Hilbert-type integral operators, yielding a unified boundedness criterion in terms
of a discrete summability condition on the kernel. As applications, we establish the boundedness of the $p$-adic Hardy operator, the $p$-adic Hilbert operator, the $p$-adic Hardy--Littlewood--P\'olya operator, and a $p$-adic analogue of the Erd\'elyi--Kober fractional integral operator \cite{KVS,SI} on fractional Ces\`aro--Morrey spaces over
$\mathbb{Q}_p$.

The paper is organized as follows. In Section~2 we recall several basic notions from $p$-adic analysis that will be needed in the sequel. Section~3 introduces fractional Ces\`aro--Morrey spaces over $\mathbb{Q}_p$ together with the associated dilation operators and their basic properties. In Section~4 we establish the main boundedness theorem for integral operators on these spaces. Section~5 is devoted to applications of this result, where we obtain $p$-adic versions of Hardy-type, Hilbert-type, and Hardy–Littlewood–Pólya inequalities. Finally, we define the $p$-adic Erd\'elyi--Kober fractional integral operator and prove its boundedness on fractional Ces\`aro--Morrey spaces. In the last section, we introduce the $p$-adic Hadamard-type fractional integral operator and establish its boundedness on the same spaces.

%%%%%%%%%%%%%%%%%%%%%%%%%%%%%%%%%%%%%%%%%%%%%%%%%%%%%%%%%%%%%%%%%%%%%%%%%%%%%%%%%%%%%%%%%%%%%%%%%%%%%%%%%

\section{Preliminaries}\label{ch4_S2}

In this section we recall several basic notions from $p$-adic analysis that will be used throughout the paper. Our purpose is mainly to establish notation and summarize those properties of the field $\mathbb{Q}_p$ that are needed for the study of fractional Ces\`aro--Morrey spaces and the associated integral operators. For further details and proofs of the results stated here we refer the reader to the standard references \cite{BJ3}.

Let $p$ be a fixed prime number. The field of $p$-adic numbers $\mathbb{Q}_p$ is obtained by completing the rational numbers $\mathbb{Q}$ with respect to the $p$-adic absolute value $|\cdot|_p$. Every nonzero rational number $x$ can be written in the form
\[
x=p^{\ell}\frac{m'}{n'},
\]
where $\ell,m',n'\in\mathbb{Z}$ and neither $m'$ nor $n'$ is divisible by $p$. The $p$-adic norm $|\cdot|_p:\mathbb{Q}\to\mathbb{R}$ is defined by
\[
|x|_p=
\begin{cases}
p^{-\ell}, & x=p^{\ell}\dfrac{m'}{n'},\\
0, & x=0 .
\end{cases}
\]

A characteristic feature of the $p$-adic norm is that it is non Archimedean. In particular, it satisfies the ultrametric inequality
\[
|x+y|_p\le \max\{|x|_p,|y|_p\}, \qquad x,y\in\mathbb{Q}_p.
\]
Moreover, whenever $|x|_p\ne |y|_p$, the larger of the two norms dominates and therefore
\[
|x+y|_p=\max\{|x|_p,|y|_p\}.
\]

Every nonzero element $x\in\mathbb{Q}_p$ admits a unique expansion of the form
\begin{equation}\label{ch4_s1}
x=p^{\ell}\sum_{j=0}^{\infty}c_jp^{j},
\end{equation}
where $c_j\in\mathbb{Z}/p\mathbb{Z}$ and $c_0\neq0$. The series in \eqref{ch4_s1} converges in the $p$-adic norm since
\[
|p^{\ell}c_jp^j|_p\le p^{-\ell-j}\to0
\quad \text{as } j\to\infty .
\]

For $a\in\mathbb{Q}_p$ and $\ell\in\mathbb{Z}$ we define the $p$-adic ball and sphere by
\[
B^{\ell}(a)=\{x\in\mathbb{Q}_p:|x-a|_p\le p^{\ell}\},
\qquad
S^{\ell}(a)=\{x\in\mathbb{Q}_p:|x-a|_p=p^{\ell}\}.
\]
When the center is the origin we simply write
\[
B^{\ell}=B^{\ell}(0), \qquad S^{\ell}=S^{\ell}(0).
\]
With this notation one has the decompositions
\[
\mathbb{Q}_p=\bigcup_{\ell=-\infty}^{\infty}B^{\ell},
\qquad
\mathbb{Q}_p^\ast=\mathbb{Q}_p\setminus\{0\}
=\bigcup_{\ell=-\infty}^{\infty}S^{\ell}.
\]

Since the additive group of $\mathbb{Q}_p$ is a locally compact Abelian group, it admits a translation invariant Haar measure $dx$. We normalize this measure by requiring that the unit ball has measure one, namely
\[
|B^0|=\int_{B^0}dx=1,
\]
where $|E|$ denotes the Haar measure of a measurable set $E\subset\mathbb{Q}_p$. With this normalization the measures of balls and spheres are given by
\[
|B^{\ell}|=p^{\ell}, \qquad |S^{\ell}|=p^{\ell}(1-p^{-1}), \qquad \ell\in\mathbb{Z}.
\]

We finally recall a useful scaling property of the Haar measure, which plays an important role in the analysis of dilation invariant operators. For any $\xi\in\mathbb{Q}_p^\ast$ and any integrable function $f:\mathbb{Q}_p\to\mathbb{C}$, the change of variables $x=\xi y$ yields
\[
\int_{\mathbb{Q}_p}f(\xi x)\,dx
=
|\xi|_p^{-1}
\int_{\mathbb{Q}_p}f(x)\,dx.
\]
% Equivalently, for every measurable set $E\subset\mathbb{Q}_p$ one has
% \[
% |\xi E|=|\xi|_p\,|E|,
% \qquad
% \xi E=\{\xi x:x\in E\}.
% \]
% In particular multiplication by $\xi$ acts as a dilation on $p$-adic balls and satisfies
% \[
% \xi B^{\ell}=B^{\ell+v_p(\xi)},
% \]
% where $v_p(\xi)$ denotes the $p$-adic valuation of $\xi$.

These basic properties of the $p$-adic norm and Haar measure will be used repeatedly in the sequel, especially in the study of dilation operators and integral operators acting on fractional Ces\`aro--Morrey spaces over $\mathbb{Q}_p$.

%%%%%%%%%%%%%%%%%%%%%%%%%%%%%%%%%%%%%%%%%%%%%%%%%%%%%%%%%%%%%%

\section{\textbf{Fractional Ces\`aro--Morrey spaces over $p$-adic field}}

In this section we introduce the fractional Ces\`aro--Morrey function spaces over the $p$-adic field $\mathbb{Q}_p$. After presenting the definition of these spaces, we investigate several basic properties that will be needed in the later sections. In particular, we study the action of dilation operators on fractional Ces\`aro--Morrey spaces and obtain the corresponding boundedness estimates. 
We also establish a Minkowski type integral inequality adapted to the present setting. These preliminary results provide the analytical tools required for the operator boundedness results that will be developed in the subsequent sections.

\begin{definition}
    Let $ 1 \leq r < \infty.$ The $p$-adic Lebesgue space  $L^r(\mathbb{Q}_p)$ consists of all measurable functions $f$ on $\mathbb{Q}_p$ such that 
\begin{align*}
\|f\|_{L^r(\mathbb{Q}_p)} = \Bigg( \int_{\mathbb{Q}_p}|f(x)|^r \di x\Bigg)^{1/r} < \infty.
\end{align*}
\end{definition}

The space $L_{\mathrm{loc}}^r(\mathbb{Q}_p)$ is defined as the set of all measurable functions $f$ on $\mathbb{Q}_p$ satisfying $\int_{K}|f(x)|^r \di x < \infty,$ for any compact subset $K$ of $\mathbb{Q}_p.$
%\begin{definition}
%Let $1 \leq r < \infty$ and if $f \in L_{\mathrm{loc}}^r(\mathbb{Q}_p)$ Then the $p$-adic Hardy operator  $\mathscr{H}^pf$ is defined by
%\begin{align*}
%\mathscr{H}^p f(x)
%=
%\frac{1}{\mu(B(0,|x|_p))} 
%\int_{B(0,|x|_p)} |f(y)|\, \di\mu(y)
%=
%\frac{1}{|x|_p} \int_{|y|_p\le |x|_p} |f(y)|\, \di\mu(y).. 
%\end{align*}
%\end{definition}
%\begin{definition}
% Let $f:\mathbb{Q}_p \to \mathbb{C}$ be a measurable function, the Copson operator
%$\mathscr{H}^{\ast,p}$ is defined by
%\[
%(\mathscr{H}^{\ast,p} f)(x)
%:=
%\int_{\{y \in \mathbb{Q}_p : |y|_p \ge |x|_p\}}
%\frac{f(y)}{|y|_p}\,dy,
%\qquad x \in \mathbb{Q}_p \setminus \{0\}.
%\]
%\end{definition}
\begin{definition} \label{Def:1} Let $1 \leq r < \infty.$ The $p$-adic Ces\`aro function spaces $\mathrm{Ces}_{r}(\mathbb{Q}_p)$ consists of all measurable functions $f$ on $\mathbb{Q}_p$ such that

%$f:\mathbb{Q}_p\to\mathbb{R}$ be measurable.  
%Define the p-adic Cesàro function space
%\begin{align*}
%Ces_{r}(\mathbb{Q}_p)
%=
%\left\{
%f:\mathbb{Q}_p\to\mathbb{R}~~ \mu-\text{ measurable} :
%\|f\|_{Ces_{r}(\mathbb{Q}_p)} < \infty
%\right\},
%\end{align*} 
%where the norm is
\begin{align} \label{Ces:norm1}
    \|f\|_{\mathrm{Ces}_{r}(\mathbb{Q}_p)}=\|\mathscr{H}^p (|f|)\|_{L^{r}(\mathbb{Q}_p)}=\left( \int_{\mathbb{Q}_p} \left( \frac{1}{|x|_p} \int_{|y|_p \leq |x|_p} |f(y)|\, \di y \right)^r \di x\right)^{1/r} < \infty.
\end{align}
\end{definition}

For any locally integrable function $f$ on $\mathbb{Q}_p,$ the $p$-adic Hardy operator is defined as 
\begin{align} \label{op:Hardy}
\mathscr{H}^p f(x) = \frac{1}{|x|_p} \int_{|y|_p \leq |x|_p} f(y) \di y,~~~  x \in \mathbb{Q}^*_p.
\end{align}
%$p$-adic Hardy inequality established in \cite{}(see Corollary 2.2) yields the embedding $L^{r}(\mathbb{Q}_p) \hookrightarrow \mathrm{Ces}_{r}(\mathbb{Q}_p),~~~~~ r \in (1, \infty).$

\begin{definition}
    Let $f$ be a locally integrable function on $\mathbb{Q}_p$ and let the order 
$0\le \beta <1$. We define the $p$-adic fractional Hardy operator by
\[
\mathscr{H}_p^{\beta}f(x)
= \frac{1}{|B(0,|x|_p)|^{1-\beta}}
\int_{B(0,|x|_p)} f(t)\,dt
= \frac{1}{|x|_p^{1-\beta}}
\int_{|t|_p\le |x|_p} f(t)\,dt,
\qquad x\in \mathbb{Q}_p^* .
\]

%It is bounded on $L^r(\mathbb{Q}_p)$ for $1<r<\infty$.
\end{definition}
  \medskip
\begin{definition}
Let $1\le r<\infty$ and $0\le\beta<1$. The fractional
Ces\`aro function space on $\mathbb{Q}_p$, denoted by
$\mathrm{Ces}_r^\beta(\mathbb{Q}_p)$, consists of all
$f\in L_{\mathrm{loc}}^r(\mathbb{Q}_p)$ such that
\[
\|\mathscr H_p^\beta(|f|)\|_{L^r(\mathbb Q_p)}<\infty.
\]
The corresponding norm is
\begin{equation}\label{FC:norm}
\|f\|_{\mathrm{Ces}_r^\beta(\mathbb Q_p)}
=
\|\mathscr H_p^\beta(|f|)\|_{L^r(\mathbb Q_p)}
=
\left(
\int_{\mathbb Q_p}
\left(
\frac{1}{|x|_p^{1-\beta}}
\int_{B(0,|x|_p)}|f(t)|\,dt
\right)^r dx
\right)^{1/r}.
\end{equation}
\end{definition}
%Hence the space $\mathrm{Ces}_r^{\beta}(\mathbb{Q}_p)$ is well defined.

\medskip

%For $a\in\mathbb{Q}_p$ and $\ell\in\mathbb{Z}$, define
%\[
%B=B_\ell(a)=\{x\in\mathbb{Q}_p:|x-a|_p\le p^\ell\},
%\]
%which denotes the ball of radius $p^\ell$ centered at $a$.

\begin{definition}
Let $1\le r\le q<\infty$. The $p$-adic Morrey space $M_{r,q}(\mathbb{Q}_p)$ consists of all measurable functions $f$ on $\mathbb{Q}_p$ such that
\[
\|f\|_{M_{r,q}(\mathbb{Q}_p)}
=
\sup_{\substack{a\in\mathbb{Q}_p\\ \ell\in\mathbb{Z}}}
|B^\ell(a)|^{\frac{1}{q}-\frac{1}{r}}
\left(
\int_{B^\ell(a)} |f(t)|^q\,dt
\right)^{\frac{1}{q}}
<\infty.
\]
\end{definition}

% Since $|B^\ell(a)|=p^\ell$, the above expression can also be written as
% \[
% \|f\|_{M_{q,r}(\mathbb{Q}_p)}
% =
% \sup_{\substack{a\in\mathbb{Q}_p \\ \ell\in\mathbb{Z}}}
% p^{\ell\left(\frac{1}{q}-\frac{1}{r}\right)}
% \left(
% \int_{B^\ell(a)} |f(x)|^{r}\,dx
% \right)^{\frac{1}{r}} .
% \]

\medskip

\begin{remark} We have following remarks.
    \begin{enumerate}
\item If $r=q$, then
\[
M_{r,r}(\mathbb{Q}_p)=L^{r}(\mathbb{Q}_p).
\]

\item If $r<q$, then
\[
L^{q}(\mathbb{Q}_p)\subset M_{r,q}(\mathbb{Q}_p)
\subset L^{r}_{\mathrm{loc}}(\mathbb{Q}_p).
\]
\end{enumerate}

\end{remark}

\medskip

Next we are going to define the  Ces\`aro--Morrey space and fractional Ces\`aro--Morrey space over $\mathbb{Q}_p$.

\begin{definition}
Let $1\le r\le q<\infty$. The $p$-adic Ces\`aro--Morrey space, denoted by $C_{M_{r,q}}(\mathbb{Q}_p)$, consists of all locally integrable functions $g$ on $\mathbb{Q}_p$ such that
\[
\|g\|_{\mathrm{Ces}\,M_{r,q}(\mathbb{Q}_p)}
:=
\|\mathscr H_p(|g|)\|_{M_{r,q}(\mathbb{Q}_p)}
<\infty.
\]
Equivalently,
\[
\|g\|_{\mathrm{Ces}\,M_{r,q}(\mathbb{Q}_p)}
=
\sup_{\substack{a\in\mathbb{Q}_p\\ \ell\in\mathbb Z}}
|B^\ell(a)|^{\frac1q-\frac1r}
\left(
\int_{B^\ell(a)}
\left(
\frac1{|B(0,|x|_p)|}
\int_{|t|_p\le |x|_p}|g(t)|\,dt
\right)^r dx
\right)^{1/r}
<\infty.
\]
\end{definition}

\begin{definition}\label{def:FCM}
Let $1\le r\le q<\infty$ and $0\le \beta<1$. The fractional
Ces\`aro--Morrey space on $\mathbb{Q}_p$, denoted by
$C^{\beta}M_{r,q}(\mathbb{Q}_p)$, consists of all locally
integrable functions $f$ on $\mathbb{Q}_p$ such that
\[
\|f\|_{C^{\beta}M_{r,q}(\mathbb{Q}_p)}
:=
\|\mathscr{H}_p^{\beta}(|f|)\|_{M_{r,q}(\mathbb{Q}_p)}
<\infty.
\]

Equivalently,
\begin{equation}\label{FCM:norm1}
\begin{aligned}
\|f\|_{C^{\beta}M_{r,q}(\mathbb{Q}_p)}
=
\sup_{\substack{a\in\mathbb{Q}_p\\ \ell\in\mathbb{Z}}}
&|B^\ell(a)|^{\frac{1}{q}-\frac{1}{r}}
\left(
\int_{B^\ell(a)}
\left(
\frac{1}{|B(0,|x|_p)|^{\,1-\beta}}
\int_{|t|_p\le |x|_p}|f(t)|\,dt
\right)^r
dx
\right)^{\frac{1}{r}}.
\end{aligned}
\end{equation}
\end{definition}
For brevity, we shall use the notation
$C_{M_{r,q}}^{\beta}(\mathbb{Q}_p)$ for the fractional
Ces\`aro--Morrey space over $\mathbb{Q}_p$.

\begin{remark}
The fractional Ces\`aro--Morrey space introduced above includes the
Ces\`aro--Morrey space, the fractional Ces\`aro space, and the
Ces\`aro function space over $\mathbb{Q}_p$ as special cases.

\begin{enumerate}
\item[(i)] If $\beta=0$, then
\[
C^{0}M_{r,q}(\mathbb{Q}_p)
=
C_{M_{r,q}}(\mathbb{Q}_p).
\]

\item[(ii)] If $r=q$, then
\[
C^{\beta}M_{r,r}(\mathbb{Q}_p)
=
\mathrm{Ces}_{r}^{\beta}(\mathbb{Q}_p).
\]
\item[(iii)] If $r=q$ and $\beta=0$, then
\[
C^{0}M_{r,r}(\mathbb{Q}_p)
=
\mathrm{Ces}_{r}(\mathbb{Q}_p).
\]

\end{enumerate}
\end{remark}

\begin{proposition}
Let $1\le r\le q<\infty$ and $0\le \beta<1$ If $f\in C_{M_{r,q}}^{\beta}(\mathbb{Q}_p)$, then
\[
\mathscr{H}_p^{\beta}(|f|)
\in L^{r}_{\mathrm{loc}}(\mathbb{Q}_p).
\]
\end{proposition}

\begin{proof}
Let $f\in C_{M_{r,q}}^{\beta}(\mathbb{Q}_p)$. Then by definition
\[
\mathscr{H}_p^{\beta}(|f|) \in M_{r,q}(\mathbb{Q}_p),
\]
and hence

\begin{align}\label{**}
    \|\mathscr{H}_p^{\beta}(|f|)\|_{M_{r,q}(\mathbb{Q}_p)}<\infty. 
\end{align}

Thus,
\[
\sup_{\substack{a\in\mathbb{Q}_p\\ \ell\in\mathbb{Z}}}
|B^\ell(a)|^{\frac{1}{q}-\frac{1}{r}}
\left(
\int_{B^\ell(a)}
\left|\mathscr{H}_p^{\beta}(|f|)(x)\right|^r\,dx
\right)^{\frac{1}{r}}
<\infty.
\]

Let $B^\ell(a)\subset\mathbb{Q}_p$ be an arbitrary $p$-adic ball. From the definition of supremum we obtain
\[
|B^\ell(a)|^{\frac{1}{q}-\frac{1}{r}}
\left(
\int_{B^\ell(a)} |\mathscr{H}_p^{\beta}(|f|)(x)|^{r}\,dx
\right)^{\frac{1}{r}}
\le
\|\mathscr{H}_p^{\beta}(|f|)\|_{M_{r,q}(\mathbb{Q}_p)} .
\]

Hence
\[
\left(
\int_{B^\ell(a)} |\mathscr{H}_p^{\beta}(|f|)(x)|^{r}\,dx
\right)^{\frac{1}{r}}
\le
|B^\ell(a)|^{\frac{1}{r}-\frac{1}{q}}
\|\mathscr{H}_p^{\beta}(|f|)\|_{M_{r,q}(\mathbb{Q}_p)} .
\]

Consequently,
\[
\int_{B^\ell(a)} |\mathscr{H}_p^{\beta}(|f|)(x)|^{r}\,dx
\le
|B^\ell(a)|^{1-\frac{r}{q}}
\|\mathscr{H}_p^{\beta}(|f|)\|_{M_{r,q}(\mathbb{Q}_p)}^{r}.
\]

By \eqref{**} it follows that
\[
\int_{B^\ell(a)} |\mathscr{H}_p^{\beta}(|f|)(x)|^{r}\,dx < \infty .
\]

Since every compact subset $K$ of $\mathbb{Q}_p$ is contained in
some $p$-adic ball $B^\ell(a)$, it follows that
\[
\int_K
\left|\mathscr{H}_p^{\beta}(|f|)(x)\right|^r\,dx
<\infty.
\]

Therefore,
\[
\mathscr{H}_p^{\beta}(|f|)
\in L^r_{\mathrm{loc}}(\mathbb{Q}_p).
\]
\end{proof}

%We now study the dilation operator on $p$-adic Ces\`aro function spaces. 
For any measurable function $f$ on $\mathbb{Q}_p$ and $\xi \in \mathbb{Q}_p^{*}$, 
the dilation operator is defined by
\begin{align*}
(D_{\xi}f)(x) = f(\xi x), \qquad x \in \mathbb{Q}_p .
\end{align*}
In the following lemmas we investigate the action of the dilation operator on fractional Ces\`aro--Morrey spaces over $\mathbb{Q}_p$.

\begin{lemma}
For every $\xi\in\mathbb{Q}_p^\ast$, the fractional Hardy operator
$\mathscr{H}_p^\beta$ satisfies
\begin{align}\label{*}
\mathscr{H}_p^\beta(|D_\xi f|)(x)
=
|\xi|_p^{-\beta}
\mathscr{H}_p^\beta(|f|)(\xi x).
\end{align}
\end{lemma}

\begin{proof}
    \[
\mathscr{H}_p^{\beta}(|D_\xi f|)(x)
=
\frac{1}{|x|_p^{1-\beta}}
\int_{|t|_p\le |x|_p}
|f(\xi t)|\,dt .
\]

Putting $y=\xi t$, the Haar measure satisfies $dt=|\xi|_p^{-1}dy$. Hence
\[
\mathscr{H}_p^{\beta}(|D_\xi f|)(x)
=
\frac{|\xi|_p^{-1} }{|x|_p^{1-\beta}}
\int_{|y|_p\le |\xi x|_p}
|f(y)|\,dy .
\]

Thus
\[
\mathscr{H}_p^{\beta}(|D_\xi f|)(x)
=
\frac{|\xi|_p^{-1}|\xi|_p^{1-\beta}}{|\xi x|_p^{1-\beta}}
\int_{|y|_p\le |\xi x|_p}
|f(y)|\,dy
=
|\xi|_p^{-\beta} \mathscr{H}_p^{\beta}| f|(\xi x).
\]

\end{proof}

%\hfill $\square$

\medskip

\begin{lemma}\label{lem:dia}
Let $1\le r\le q<\infty$ and $0\le \beta<1$. For every
$\xi\in\mathbb{Q}_p^\ast$, the dilation operator $D_{\xi}$ is
bounded on $C^{\beta}M_{r,q}(\mathbb{Q}_p)$. Moreover,
\[
\|D_{\xi}f\|_{C^{\beta}M_{r,q}(\mathbb{Q}_p)}
=
|\xi|_p^{-\beta-\frac{1}{q}}
\|f\|_{C^{\beta}M_{r,q}(\mathbb{Q}_p)}.
\]
\end{lemma}

\begin{proof}
Let $B=B^\ell(a)\subset\mathbb{Q}_p$. Then
\[
|B|^{\frac{1}{q}-\frac{1}{r}}
\left(
\int_B
\left|\mathscr{H}_p^{\beta}(|D_{\xi}f|)(x)\right|^r dx
\right)^{\frac{1}{r}} .
\]

Using \eqref{*}, we obtain
\[
=
|\xi|_p^{-\beta}
|B|^{\frac{1}{q}-\frac{1}{r}}
\left(
\int_B
\left|\mathscr{H}_p^{\beta}(|f|)(\xi x)\right|^r dx
\right)^{\frac{1}{r}} .
\]

Substituting $z=\xi x$ and using the scaling property of the Haar
measure, $dx=|\xi|_p^{-1}dz$, we obtain
\[
=
|\xi|_p^{-\beta}
|B|^{\frac{1}{q}-\frac{1}{r}}
\left(
\int_{\xi B}
\left|\mathscr{H}_p^{\beta}(|f|)(z)\right|^r
|\xi|_p^{-1}dz
\right)^{\frac{1}{r}} .
\]

Thus,
\[
=
|\xi|_p^{-\beta-\frac{1}{r}}
|B|^{\frac{1}{q}-\frac{1}{r}}
\left(
\int_{\xi B}
\left|\mathscr{H}_p^{\beta}(|f|)(z)\right|^r dz
\right)^{\frac{1}{r}} .
\]

Since $|\xi B|=|\xi|_p|B|$, we have
\[
|B|^{\frac{1}{q}-\frac{1}{r}}
=
|\xi|_p^{-\frac{1}{q}+\frac{1}{r}}
|\xi B|^{\frac{1}{q}-\frac{1}{r}}.
\]

Hence,
\[
=
|\xi|_p^{-\beta-\frac{1}{q}}
|\xi B|^{\frac{1}{q}-\frac{1}{r}}
\left(
\int_{\xi B}
\left|\mathscr{H}_p^{\beta}(|f|)(z)\right|^r dz
\right)^{\frac{1}{r}} .
\]

By the definition of the norm in \eqref{FCM:norm1},
\[
\leq
|\xi|_p^{-\beta-\frac{1}{q}}
\|f\|_{C^{\beta}M_{r,q}(\mathbb{Q}_p)}.
\]

Taking the supremum over all balls $B\subset\mathbb{Q}_p$, we obtain
\begin{align}\label{5}
\|D_{\xi}f\|_{C^{\beta}M_{r,q}(\mathbb{Q}_p)}
\leq
|\xi|_p^{-\beta-\frac{1}{q}}
\|f\|_{C^{\beta}M_{r,q}(\mathbb{Q}_p)}.
\end{align}

Since $D_{\xi^{-1}}D_{\xi}f=f$, applying \eqref{5} with
$\xi^{-1}$ in place of $\xi$ gives
\begin{align}\label{6}
\|f\|_{C^{\beta}M_{r,q}(\mathbb{Q}_p)}
&=
\|D_{\xi^{-1}}D_{\xi}f\|_{C^{\beta}M_{r,q}(\mathbb{Q}_p)}\\
&\leq
|\xi|_p^{\beta+\frac{1}{q}}
\|D_{\xi}f\|_{C^{\beta}M_{r,q}(\mathbb{Q}_p)}.
\end{align}

Therefore,
\[
\|D_{\xi}f\|_{C^{\beta}M_{r,q}(\mathbb{Q}_p)}
\geq
|\xi|_p^{-\beta-\frac{1}{q}}
\|f\|_{C^{\beta}M_{r,q}(\mathbb{Q}_p)}.
\]

Combining this with \eqref{5}, we conclude that
\[
\|D_{\xi}f\|_{C^{\beta}M_{r,q}(\mathbb{Q}_p)}
=
|\xi|_p^{-\beta-\frac{1}{q}}
\|f\|_{C^{\beta}M_{r,q}(\mathbb{Q}_p)}.
\]
\end{proof}

%\hfill $\square$

\begin{lemma}\label{lem:minkowski}
Let $1\le r\le q<\infty$ and $0\le\beta<1$. For any Haar measurable
function
$F:\mathbb{Q}_p\times\mathbb{Q}_p\to\mathbb{C}$ such that
\[
\int_{\mathbb{Q}_p}
\|F(y,\cdot)\|_{C^\beta M_{r,q}(\mathbb{Q}_p)}\,dy<\infty,
\]
we have
\[
\left\|
\int_{\mathbb{Q}_p}F(y,\cdot)\,dy
\right\|_{C^\beta M_{r,q}(\mathbb{Q}_p)}
\le
\int_{\mathbb{Q}_p}
\|F(y,\cdot)\|_{C^\beta M_{r,q}(\mathbb{Q}_p)}\,dy.
\]
\end{lemma}

\begin{proof}
    Let $B=B^\ell(a)\subset \mathbb{Q}_p$ be any ball, we have

\[
|B|^{\frac{1}{q}-\frac{1}{r}}
\left(
\int_B
\left(
\frac{1}{|x|_p^{1-\beta}}
\int_{|t|_p\le |x|_p}
\left|
\int_{\mathbb{Q}_p}F(y,t)\,dy
\right|
dt
\right)^r
dx
\right)^{\frac{1}{r}}
\]

\[
\le
|B|^{\frac{1}{q}-\frac{1}{r}}
\left(
\int_B
\left(
\frac{1}{|x|_p^{1-\beta}}
\int_{|t|_p\le |x|_p}
\int_{\mathbb{Q}_p}|F(y,t)|\,dy\,dt
\right)^r
dx
\right)^{\frac{1}{r}}
\]

\[
=
|B|^{\frac{1}{q}-\frac{1}{r}}
\left(
\int_B
\left(
\int_{\mathbb{Q}_p}
\frac{1}{|x|_p^{1-\beta}}
\int_{|t|_p\le |x|_p}
|F(y,t)|\,dt\,dy
\right)^r
dx
\right)^{\frac{1}{r}}
\qquad (\text{By Tonelli's theorem})
\]

\[
\le
|B|^{\frac{1}{q}-\frac{1}{r}}
\int_{\mathbb{Q}_p}
\left(
\int_B
\left(
\frac{1}{|x|_p^{1-\beta}}
\int_{|t|_p\le |x|_p}
|F(y,t)|\,dt
\right)^r
dx
\right)^{\frac{1}{r}}
dy
\]

here we use the Minkowski inequality for integrals \cite{BJ3} (see Theorem 1.42) in the last inequality.

Now, according to Definition \eqref{def:FCM}, we have

\[
|B|^{\frac{1}{q}-\frac{1}{r}}
\left(
\int_B
\left(
\frac{1}{|x|_p^{1-\beta}}
\int_{|t|_p\le |x|_p}
\left|
\int_{\mathbb{Q}_p}F(y,t)\,dy
\right|
dt
\right)^r
dx
\right)^{\frac{1}{r}}
\]

\[
\le
\int_{\mathbb{Q}_p}
|B|^{\frac{1}{q}-\frac{1}{r}}
\left(
\int_B
\left(
\frac{1}{|x|_p^{1-\beta}}
\int_{|t|_p\le |x|_p}
|F(y,t)|\,dt
\right)^r
dx
\right)^{\frac{1}{r}}
dy
\]

\[
\le
\int_{\mathbb{Q}_p}
\|F(y,\cdot)\|_{C^{\beta}M_{r,q}(\mathbb{Q}_p)}\,dy
\qquad 
\]

By taking the supremum over all balls $B\subset\mathbb{Q}_p$ on both sides,
we obtain

\[
\left\|
\int_{\mathbb{Q}_p}F(y,\cdot)\,dy
\right\|_{C^{\beta}M_{r,q}(\mathbb{Q}_p)}
\le
\int_{\mathbb{Q}_p}
\|F(y,\cdot)\|_{C^{\beta}M_{r,q}(\mathbb{Q}_p)}\,dy .
\]

\end{proof}

%In the following lemmas, we collect several fundamental properties of the
%Ces\`aro function space $\mathrm{Ces}_r(\mathbb{Q}_p)$. In particular, we
%establish the boundedness of the dilation operator and prove a
%Minkowski-type inequality in this setting.

%%%%%%%%%%%%%%%%%%%%%%%%%%%%%%%%%%%%%%%%%%%%%%%%%%%%%%%%%%%%%%%%%%%%%%%%%%%%%%%%%%%%%%%%%%%%%%%%%%	
\section{$p$-adic Hardy--Hilbert-type integral operators on Ces\`aro function spaces}
We now turn to the study of integral operators induced by homogeneous kernel acting on $p$-adic Ces\`aro function spaces. The Hardy--Hilbert-type integral operator $\mathscr{T}^p,$ over the $p$-adic field, is defined as 
\begin{align}\label{tf:op}
\mathscr{T}^pf(x) = \int_{\mathbb{Q}_p^*} \mathscr{K}\bigl(|x|_p,|y|_p\bigr) f(y) \di y,\qquad x \in \mathbb{Q}_p^*.
\end{align}
where $\mathscr{K}:(0,\infty)\times(0,\infty)\to[0,\infty)$ is a measurable function satisfying
\begin{align} \label{Homoge}
    \mathscr{K}(\tau x,\tau y) = \tau^{-1}\mathscr{K}(x,y), \qquad \tau >0.
\end{align}

The operator $\mathscr{T}^p$ was first introduced by Li and Jin \cite{Li}, who established its boundedness on weighted Lebesgue spaces. Further developments have extended these boundedness results to several other function spaces, such as $p$-adic block spaces, two-weighted Morrey spaces, Morrey--Herz spaces, and weighted Triebel--Lizorkin spaces; see \cite{SH1,KH1,KH2}.

The following theorem establishes the boundedness result for $p$-adic Hardy--Hilbert-type integral operators on Ces\`aro function spaces.

\begin{theorem} \label{thm:main}
Let Let $1\le r\le q<\infty$ and $0\le \beta <1$. If the kernel $\mathscr{K}$ satisfies 
\begin{equation} \label{thm:main1}
C_{p,\beta,q}:=(1-p^{-1})\sum_{j\in\mathbb{Z}}\mathscr{K}(1,p^j)\,p^{j(1-\beta -1/q)}  < \infty.
\end{equation}
Then, for any $f\in \mathrm{C^{\beta}M_{r,q}(\mathbb{Q}_p)}$, we have
\begin{align*}
    \|\mathscr{T}^pf\|_{C^{\beta}M_{r,q}(\mathbb{Q}_p)}\leq C_{p,\beta,q}\|f\|_{C^{\beta}M_{r,q}(\mathbb{Q}_p)},
\end{align*}
where $\mathscr{T}^p$ is the operator in \eqref{tf:op}.
\end{theorem}
\begin{proof} 
After the change of variable $y=\xi x$ and using $\di y=|x|_p\di \xi$
\begin{align*}
\mathscr{T}^pf(x) &= \int_{\mathbb{Q}_p^*} \mathscr{K}\bigl(|x|_p,|\xi x|_p\bigr) f(\xi x) |x|_p \di \xi \\
&= \int_{\mathbb{Q}_p^*} \mathscr{K}\bigl(|x|_p,|\xi|_p|x|_p\bigr) f(\xi x) |x|_p \di \xi \\
&= \int_{\mathbb{Q}_p^*}|x|^{-1}_p \mathscr{K}\bigl(1,|\xi|_p\bigr) f(\xi x) |x|_p \di \xi \\
&= \int_{\mathbb{Q}_p^*} \mathscr{K}\bigl(1,|\xi|_p\bigr) D_\xi f( x) \di \xi. \\
\end{align*}
Applying $\|\cdot\|_{C^{\beta}M_{r,q}(\mathbb{Q}_p)}$ and by using Lemma \ref{lem:dia} and Lemma \ref{lem:minkowski}, we get 
\begin{align} \nonumber
  \|\mathscr{T}^pf\|_{C^{\beta}M_{r,q}(\mathbb{Q}_p)} &\leq \int_{\mathbb{Q}_p^*} \mathscr{K}\bigl(1,|\xi|_p\bigr) \|D_\xi f( x)\|_{C^{\beta}M_{r,q}(\mathbb{Q}_p)} \di \xi \\ \nonumber
   &= \int_{\mathbb{Q}_p^*} \mathscr{K}\bigl(1,|\xi|_p\bigr) |\xi|^{-\beta-1/q}_p\|f\|_{C^{\beta}M_{r,q}(\mathbb{Q}_p)} \di \xi \\ \label{main:eq1}
   &= \|f\|_{C^{\beta}M_{r,q}(\mathbb{Q}_p)}\int_{\mathbb{Q}_p^*} \mathscr{K}\bigl(1,|\xi|_p\bigr) |\xi|^{-\beta-1/q}_p \di \xi. 
\end{align}
Since $\mathbb{Q}_p^*$ can be written as a disjoint union $\bigcup_{j=-\infty}^{\infty} S^j$ where $S^j=\{x\in\mathbb{Q}_p : |x|_p = p^j\}$ and using the fact that $|S^j| = p^j(1-p^{-1}),$ we have 

\begin{align} \nonumber
\int_{\mathbb{Q}_p^*} \mathscr{K}\bigl(1,|\xi|_p\bigr) |\xi|^{-\beta-1/q}_p \di \xi &= \sum_{j\in\mathbb{Z}} \int_{S^j} \mathscr{K}\bigl(1,|\xi|_p\bigr) |\xi|^{-\beta-1/q}_p \di \xi\\ \nonumber
&= \sum_{j\in\mathbb{Z}} \mathscr{K}\bigl(1,p^j\bigr) p^{-j(\beta+1/q)} |S^j|\\ \label{main:eq2}
&=(1-p^{-1})\sum_{j\in\mathbb{Z}} \mathscr{K}(1,p^j)\,p^{j(1-\beta -1/q)}.
\end{align}
Therefore, by \eqref{main:eq1} and \eqref{main:eq2}, we get
\begin{align*}
    \|\mathscr{T}^pf\|_  \mathrm{C^{\beta}M_{r,q}(\mathbb{Q}_p)}\leq C_{p,\beta,q}\|f\|_  \mathrm{C^{\beta}M_{r,q}(\mathbb{Q}_p)},
\end{align*}
which completes the proof.
\end{proof}

\section{Some particular results}
In this section, we will establish the Hardy inequality, Hilbert inequality and Hardy--Littlewood--P\'{o}lya inequality on fractional Ces\`aro--Morrey space $C^{\beta}M_{r,q}(\mathbb{Q}_p)$, by choosing particular kernel in \eqref{tf:op}. 
%Theorem \eqref{thm:main} yields the Hardy’s inequalities and  Hilbert inequality on $\mathrm{Ces}_r(\mathbb{Q}_p)$ in the following lemmas. 

If we choose the kernel $\mathscr{K}(|x|_p,|y|_p)=|x|_p^{-1}\chi_{\{y\in\mathbb{Q}_p: |y|_p \leq |x|_p \}},$ then the operator $\mathscr{T}^p$ reduces to the $p$-adic Hardy operator defined in \eqref{op:Hardy}. Hence, the following result shows that the Hardy operator $\mathscr{H}^p$ is bounded on the fractional Ces\`aro--Morrey space $C^{\beta}M_{r,q}(\mathbb{Q}_p)$, and gives an explicit value of the constant.

\begin{theorem}
Let Let $1\le r\le q<\infty$ and $0\le \beta <1-\frac{1}{q}$. Then, for any 
$f \in C^{\beta}M_{r,q}(\mathbb{Q}_p)$, we have
\begin{align*}
\|\mathscr{H}^p f\|_{C^{\beta}M_{r,q}(\mathbb{Q}_p)} 
\le C_{p,\beta,q} \, \|f\|_{C^{\beta}M_{r,q}(\mathbb{Q}_p)},
\end{align*}
where $C_{p,\beta,q}$ is given by
\begin{align*}
C_{p,\beta,q}
=
(1-p^{-1}) \sum_{j\in\mathbb{Z}} \mathscr{K}(1,p^j)\, p^{j\left(1-\beta-\frac{1}{q}\right)}.
\end{align*}

%\item
%%\[
%\|\mathscr{H}^{*,p} f\|_{\mathrm{Ces}_r(\mathbb{Q}_p)}
%\le
%(1-p^{-1})
%\frac{p^{-1/r}}{1-p^{-1/r}},
%\|f\|_{\mathrm{Ces}_r(\mathbb{Q}_p)}.
%\]
%\end{enumerate}
\end{theorem}
\begin{proof}
Let $\mathscr{K}(|x|_p,|y|_p)=|x|_p^{-1}\chi_{\{\,y\in\mathbb{Q}_p:\,|y|_p\le |x|_p\,\}}.$ Clearly, kernel satisfies the homogeneity condition \eqref{Homoge} and
\begin{align*}
    C_{p,\beta, q}&=(1-p^{-1}) \sum_{j\in\mathbb{Z}} \mathscr{K}(1,p^j)\, p^{j\left(1-\beta-\frac{1}{q}\right)} \\
    &=(1-p^{-1})\sum_{j\leq 0} p^{j\left(1-\beta-\frac{1}{q}\right)} \\
   & =(1-p^{-1})\sum_{j=0}^{\infty} p^{-j\left(1-\beta-\frac{1}{q}\right)} \\
   &= \frac{1-p^{-1}}{1 - p^{-\left(1-\beta-\frac{1}{q}\right)}}, \qquad \because \quad  0\le \beta <1-\frac{1}{q}.
\end{align*}
Therefore, result follows from Theorem \ref{thm:main}. %\textbf{add remark here to connect with previous work}
%Similarly taking
%\[
%\mathscr{K}(|s|_p,|t|_p)
%=
%|t|_p^{-1}\,
%\chi_{\{\,t\in\mathbb{Q}_p:\,|t|_p\ge |s|_p\,\}},
%\]
%in
%\[
%\mathscr{T}^pf(s)
%=
%\int_{\mathbb{Q}_p}
%\mathscr{K}(|s|_p,|t|_p)\,f(t)\,\di (t),
%\]
%we obtain
%\[
%\mathscr{H}^{*,p} f(s)
%=
%\int_{\mathbb{Q}_p}
%\mathscr{K}(|s|_p,|t|_p)\,f(t)\,\di (t).
%\]

%It satisfies the assumption  \eqref{thm:main1} and
%\[
%C_{p,r}
%=
%(1-p^{-1})
%\sum_{j\in\mathbb{Z}}
%\mathscr{K}(p^j,1)\,
%p^{\,j\left(1-\frac{1}{r}\right)}
%=
%(1-p^{-1})
%\sum_{j=1}^{\infty}
%p^{-j/r}.
%\]

%This is a convergent geometric series since $p^{-1/r}<1$ and
%\[
%C_{p,r}
%=
%(1-p^{-1})
%\frac{p^{-1/r}}{1-p^{-1/r}}
%<\infty,
%\quad r>1.
%\]
\end{proof}
If we choose
\begin{align}\label{Hilbert}
    \mathscr{K}(s,t)=\frac{1}{|s|_p+|t|_p}
\end{align}

in \eqref{tf:op}, then the operator $\mathscr{T}^p$ reduces to the $p$-adic Hilbert
operator. More precisely, for a measurable function $f$ defined on $\mathbb{Q}_p^\ast$, we
have
\[
\mathscr{H}f(t)
=
\int_{\mathbb{Q}_p^\ast}
\frac{f(s)}{|s|_p+|t|_p}\,\di (s),
\qquad t\in\mathbb{Q}_p^\ast.
\]
In the next theorem, we establish a Hilbert-type inequality for this
operator on the fractional Ces\`aro function spaces
$\mathrm{C^{\beta}M_{r,q}(\mathbb{Q}_p)}$.

\medskip

\begin{theorem}
    Let Let $1\le r\le q<\infty$ and $0\le \beta<1-\frac{1}{q}$. For any $f\in \mathrm{C^{\beta}M_{r,q}(\mathbb{Q}_p)}$, we have
\[
\|\mathscr{H}f\|_{C^{\beta}M_{r,q}(\mathbb{Q}_p)}
\le C_{p,\beta,q} \|f\|_{C^{\beta}M_{r,q}(\mathbb{Q}_p)}.
\]
\end{theorem}

\medskip

\begin{proof}
Observe that the kernel \eqref{Hilbert} is homogeneous of degree -1 and 
\begin{align*}
C_{p,\beta,q}
&=
(1-p^{-1})
\sum_{j\in\mathbb{Z}}
\mathscr{K}(1,p^j)\,p^{j\left(1-\beta-1/q\right)}
\\
&=
(1-p^{-1})
\sum_{j\in\mathbb{Z}}
\frac{p^{j(1-\beta -\frac{1}{q})}}{1+p^j}.
\end{align*}

According to \eqref{thm:main}, It is enough to show that $C_{p,\beta,q}<\infty$. Now, we can write the sum as

\begin{align*}
C_{p,\beta,q}
&= (1-p^{-1})\left(\sum_{j>0} \frac{p^{j(1-\beta -\frac{1}{q})}}{1+p^{j}}+
\sum_{j\le 0} \frac{p^{j(1-\beta-\frac{1}{q})}}{1+p^{j}}\right) \\
&= (1-p^{-1})\left(\sum_{j=1}^{\infty} \frac{p^{j(1-\beta-\frac{1}{q})}}{1+p^{j}}+\sum_{j=-\infty}^{0} \frac{p^{j(1-\beta-\frac{1}{q})}}{1+p^{j}}
\right) \\
&= (1-p^{-1})\left(\sum_{j=1}^{\infty} \frac{p^{j(1-\beta-\frac{1}{q})}}{1+p^{j}}+\sum_{j=0}^{\infty} \frac{p^{-j(1-\beta-\frac{1}{q})}}{1+p^{-j}}
\right) \\
&= (1-p^{-1})\left(\sum_{j=1}^{\infty} \frac{p^{j(1-\beta-\frac{1}{q})}}{1+p^{j}}+\sum_{j=0}^{\infty} \frac{p^{j(\beta+1/q)}}{1+p^j}
\right) \\
&\le (1-p^{-1})\left(\sum_{j=1}^{\infty} \frac{p^{j(1-\beta-\frac{1}{q})}}{p^{j}}+\sum_{j=0}^{\infty} \frac{p^{j(\beta+\frac{1}{q})}}{p^j}
\right) \\
& \le (1-p^{-1})\left(\sum_{j=1}^{\infty} p^{-j(\beta+\frac{1}{q})}+\sum_{j=0}^{\infty} p^{-j(1-\beta-\frac{1}{q})}\right).
\end{align*}

Both series geometric converge because  $q>1$ and $0\le \beta<1-\frac{1}{q}$, and hence  $C_{p,\beta,q}<\infty$. Hence, using Theorem \eqref{thm:main}, we obtain the Hilbert inequality  on $\mathrm{C^{\beta}M_{r,q}(\mathbb{Q}_p)}$.
\end{proof}

 Consider the kernel in \eqref{tf:op}
\[
\mathscr{K}(|s|_p,|t|_p)
=
\frac{(|s|_p|t|_p)^{\lambda/2}}
{\max\{|s|_p,|t|_p\}^{\lambda+1}},
\qquad \lambda \ge 0.
\]
The corresponding integral operator is given by
\begin{equation}\label{ch4_polya}
\mathscr{D}^p_\lambda f(s)
=
\int_{\mathbb{Q}_p^\ast}  
\frac{(|s|_p|t|_p)^{\lambda/2}}
{\max\{|s|_p,|t|_p\}^{\lambda+1}}
\,f(t)\,\di (t),
\qquad s\in\mathbb{Q}_p^\ast .
\end{equation}
In the special case $\lambda=0$, the operator $\mathscr{D}^p_\lambda$
coincides with the $p$-adic Hardy--Littlewood--P\'{o}lya operator
$\mathscr{P}^p$, which is defined by
\[
\mathscr{P}^p f(s)
=
\int_{\mathbb{Q}_p^\ast}
\frac{f(t)}{\max\{|s|_p,|t|_p\}}
\,\di (t),
\qquad s\in\mathbb{Q}_p^\ast .
\]

 Finally, as an application of the main boundedness theorem, we study the
operator $\mathscr{D}^p_\lambda$ and establish its boundedness on
$\mathrm{C^{\beta}M_{r,q}(\mathbb{Q}_p)}$.

\begin{theorem}
Suppose Let $1\le r\le q<\infty$ and $\lambda/2+1>\beta+1/q$. Then the operator
$\mathscr{D}^p_\lambda$, defined by \eqref{ch4_polya}, is bounded on
$\mathrm{C^{\beta}M_{r,q}(\mathbb{Q}_p)}$.
\end{theorem}
\begin{proof}
 As in the proof of Theorem \ref{thm:main}, it is sufficient to verify the finiteness of the constant $C_{p,\beta,q}$.

For $a\in\mathbb{Q}_p^*$, we have
\[
\mathscr{K}(|as|_p,|at|_p)
=
\frac{(|a|_p|s|_p\,|a|_p|t|_p)^{\lambda/2}}
{\max\{|a|_p|s|_p,|a|_p|t|_p\}^{\lambda+1}}
=
|a|_p^{-1}\mathscr{K}(|s|_p,|t|_p),
\]
which shows that the  kernel \eqref{ch4_polya}  satisfies the required homogeneity condition \eqref{Homoge}.

Next, we compute
\begin{align*}
C_{p,\beta,q}=(1-p^{-1})\sum_{j\in\mathbb{Z}} \mathscr{K}(1,p^j)\,p^{j\left(1-\beta-1/q\right)}.
\end{align*}
A direct computation yields
\begin{align*}
\mathscr{K}(1,p^j)=\frac{p^{j\lambda/2}}{\max\{p^j,1\}^{\lambda+1}}=
\begin{cases}
p^{-j(\lambda/2+1)}, & j\ge0,\\[4pt]
p^{j\lambda/2}, & j<0.
\end{cases}
\end{align*}
Therefore,
\begin{align*}
C_{p,\beta,q}&=(1-p^{-1})\left(\sum_{j \ge 0}\mathscr{K}(1,p^j)\,p^{j(1-\beta-\frac{1}{q})}+\sum_{j<0}\mathscr{K}(1,p^j)\,p^{j(1-\beta-\frac{1}{q})}\right)\\
&=(1-p^{-1})\left(\sum_{j \ge 0}p^{-j(\lambda/2+1)}\,p^{j(1-\beta-\frac{1}{q})}+\sum_{j<0}p^{j\lambda/2}\,p^{j(1-\beta-\frac{1}{q})}\right)\\
&=(1-p^{-1})\left(\sum_{j\ge0}p^{-j(\lambda/2+\beta+\frac{1}{q})}+\sum_{j<0}
p^{j(\lambda/2+1-\beta-\frac{1}{q})}
\right).
\end{align*}
Since the first summation is a geometric series, it converges immediately. For the second summation, the convergence is ensured by the condition $\lambda/2 +1 >\beta+ 1/q$, which guarantees that the corresponding exponent is strictly negative. Hence, the series is summable, and consequently we obtain $C_{p,\beta,q}<\infty$. The desired boundedness then follows directly from Theorem ~\eqref{thm:main}.

\end{proof}
\begin{remark}
    Setting $\lambda=0$ in the above theorem, the operator $\mathscr{D}^p_\lambda$ reduces to the $p$-adic Hardy--Littlewood--P\'{o}lya operator $\mathscr{P}^p$, and consequently $\mathscr{P}^p$ is bounded on the fractional Ces\`aro--Morrey space $C^{\beta}M_{r,q}(\mathbb{Q}_p)$.
\end{remark}

%%%%%%%%%%%%%%%%%%%%%%%%%%%%%%%%%%%%%%%%%%%%%%%%%%%%%%%%%%%%%%%%%%%%%%%%%%%%%%
\section{$p$-adic Fractional integrals}

In this section, we introduce the $p$-adic analogues of the Erd\'elyi--Kober fractional integral operators of the first and second kinds. As an application of Theorem~\ref{thm:main}, we establish the boundedness of these operators on the fractional Ces\`aro--Morrey space over $\mathbb{Q}_p$. For the classical definitions and standard notation of Erd\'elyi--Kober fractional integral operators, we refer the reader to \cite{HO2}.

Let $\nu>0$, $\delta>0$, and $\gamma\in \mathbb{R}$. For a
measurable function $f$ on $\mathbb{Q}_p$, the first and second kind $p$-adic Erd\'elyi--Kober fractional integrals, respectively, are defined as follows:
\begin{align*}
I_{\gamma,\delta}^{\nu}f(s)=\int_{\mathbb{Q}_p^*}I(|s|_p,|t|_p) f(t) \di t, \qquad  s \in \mathbb{Q}_p^*, 
\end{align*}
and
\begin{align*}
J_{\gamma,\delta}^{\nu}f(s)=\int_{\mathbb{Q}_p^*}J(|s|_p,|t|_p) f(t) \di t, \qquad s\in\mathbb{Q}_p^*,
\end{align*}
where the kernels $I(|s|_p,|t|_p)$ and $J(|s|_p,|t|_p)$ are given by
\begin{align}
I(|s|_p,|t|_p)=\frac{|s|_p^{-\nu(\gamma+\delta)}}{\Gamma(\delta)}\chi_{\{t \in \mathbb{Q}_p : |t|_p < |s|_p\}}(t)
|t|_p^{\nu(\gamma+1)-1}
\left(|s|_p^\nu-|t|_p^\nu\right)^{\delta-1},
\end{align}
and
\begin{align}
J(|s|_p,|t|_p)=\frac{|s|_p^{\nu\gamma}}{\Gamma(\delta)}
\,\chi_{\{t \in \mathbb{Q}_p : |t|_p > |s|_p\}}(t)|t|_p^{-\nu(\gamma+\delta)+\nu-1}
\left(|t|_p^\nu-|s|_p^\nu\right)^{\delta-1}.   
\end{align}
The Erd\'elyi--Kober fractional integrals form an important class of operators in fractional calculus and arise naturally in several problems related to mathematical physics. For further details concerning their properties and applications, we refer the reader to \cite{KVS,SI}. In the following theorem, we establish the boundedness of the $p$-adic Erd\'elyi--Kober fractional integral operator on the fractional Ces\`aro--Morrey space over $\mathbb{Q}_p$.

\medskip

\begin{theorem} Let Let $1\le r\le q<\infty$ and $ \delta,\nu \in \mathbb{R}_{>0}, ~  \gamma \in \mathbb{R},~ \beta\in [0,1)$.

\begin{itemize}
\item[(1)]
If  $\nu(\gamma+1)>\beta+\frac{1}{q}$ , then there exists a constant $C>0$ such that
for any $f\in \mathrm{C^{\beta}M_{r,q}(\mathbb{Q}_p)}$,
\[
\|I_{\gamma,\delta}^{\nu} f\|_\mathrm{C^{\beta}M_{r,q}(\mathbb{Q}_p)}
\le
C\,\|f\|_\mathrm{C^{\beta}M_{r,q}(\mathbb{Q}_p)}.
\]

\item[(2)]
If $\nu\gamma +\beta >-\frac{1}{q}$, then there exists a constant $C>0$ such that
for any  $f\in \mathrm{C^{\beta}M_{r,q}(\mathbb{Q}_p)}$,
\[
\|J_{\gamma,\delta}^{\nu}  f\|_\mathrm{C^{\beta}M_{r,q}(\mathbb{Q}_p)}
\le
C\,\|f\|_\mathrm{C^{\beta}M_{r,q}(\mathbb{Q}_p)}.
\]
\end{itemize}
\end{theorem}
\medskip

\noindent

\begin{proof} 
Let us sketch the proof, since 
\begin{align*}
I_{\gamma,\delta}^{\nu}f(s)=\int_{\mathbb{Q}_p^*}I(|s|_p,|t|_p) f(t) \di t, \qquad  s \in \mathbb{Q}_p^*, 
\end{align*}
where
\begin{align*}
I(|s|_p,|t|_p)=\frac{|s|_p^{-\nu(\gamma+\delta)}}{\Gamma(\delta)}\chi_{\{t \in \mathbb{Q}_p : |t|_p < |s|_p\}}(t)
|t|_p^{\nu(\gamma+1)-1}
\left(|s|_p^\nu-|t|_p^\nu\right)^{\delta-1}.
\end{align*}
For any $a \in \mathbb{Q}_p^*,$ we have
\begin{align*}
I(|as|_p,|at|_p) &=\frac{|as|_p^{-\nu(\gamma+\delta)}}{\Gamma(\delta)}\chi_{\{t \in \mathbb{Q}_p : |t|_p < |s|_p\}}(at)
|at|_p^{\nu(\gamma+1)-1}
\left(|as|_p^\nu-|at|_p^\nu\right)^{\delta-1} \\
&=|a|_p^{-1} I(|s|_p,|t|_p).
\end{align*}
Thus, the kernel $I(|s|_p,|t|_p)$ satisfies the homogeneity condition \eqref{Homoge}. Furthermore, by Theorem \ref{thm:main}, it suffices to verify that $C_{p,\beta,q}<\infty$ in order to establish the boundedness of $I_{\gamma,\delta}^{\nu} $ on $\mathrm{C^{\beta}M_{r,q}(\mathbb{Q}_p)}$. Hence 
\begin{align*}
C_{p,\beta,q}
&=(1-p^{-1})\sum_{j\in\mathbb{Z}} I(1,p^j)\,p^{j(1-\beta -1/q)} \\
&=\frac{(1-p^{-1})}{\Gamma(\delta)}
\sum_{j<0}
(p^j)^{\nu(\gamma+1)-1}
(1-p^{j\nu})^{\delta-1}
p^{j\left(1-\beta-\frac{1}{q}\right)}\\
&=\frac{(1-p^{-1})}{\Gamma(\delta)}
\sum_{j=-\infty}^{-1}
(p^j)^{\nu(\gamma+1)-\beta-\frac{1}{q}}
(1-p^{j\nu})^{\delta-1}\\
&=\frac{(1-p^{-1})}{\Gamma(\delta)}
\sum_{j=1}^{\infty}
p^{-j\left(\nu(\gamma+1)-\beta-\frac{1}{q}\right)}
(1-p^{-j\nu})^{\delta-1}.
\end{align*}

Observe that for large $j$, $1-p^{-j\nu}\approx 1$, and hence

\[
C_{p,\beta,q}
=
\frac{(1-p^{-1})}{\Gamma(\delta)}
\sum_{j=1}^{\infty}
p^{-j\left(\nu(\gamma+1)-\beta-\frac{1}{q}\right)}.
\]

Since $\nu(\gamma+1)>\beta+\frac{1}{q}$, it follows that $C_{p,\beta,q}$ is finite. Hence, by Theorem~\ref{thm:main}, the operator $I^{\nu}_{p,\delta}$ is bounded on $C^{\beta}M_{r,q}(\mathbb{Q}_p)$.
\medskip

Similarly, the kernel
\begin{align}
J(|s|_p,|t|_p)=\frac{|s|_p^{\nu\gamma}}{\Gamma(\delta)}
\,\chi_{\{t \in \mathbb{Q}_p : |t|_p > |s|_p\}}(t)|t|_p^{-\nu(\gamma+\delta)+\nu-1}
\left(|t|_p^\nu-|s|_p^\nu\right)^{\delta-1}, 
\end{align}
satisfies the homogeneity condition \eqref{Homoge}. Moreover,
\begin{align*}
C_{p,\beta,q}
&=(1-p^{-1})\sum_{j\in\mathbb{Z}} J(1,p^j)\, p^{j\left(1-\beta-\frac{1}{q}\right)}\\
&=\frac{(1-p^{-1})}{\Gamma(\delta)}
\sum_{j>0}
\left(p^{j\nu}-1\right)^{\delta-1}
p^{j\left(\nu(\gamma+1)-\nu+1\right)}
p^{j\left(1-\beta-\frac{1}{q}\right)}\\
&=\frac{(1-p^{-1})}{\Gamma(\delta)}
\sum_{j=1}^{\infty}
p^{j\nu(\delta-1)}
p^{-j\left(\nu(\gamma+\delta)-\nu+1\right)}
p^{j\left(1-\beta-\frac{1}{q}\right)}.
\end{align*}

As for large $j$, we have $p^{j\nu}-1 \sim p^{j\nu}$. Therefore
\begin{align*}
C_{p,\beta,q}
&=\frac{(1-p^{-1})}{\Gamma(\delta)}
\sum_{j=1}^{\infty}
p^{j\left(\nu\delta-\nu-\nu \gamma-\nu\delta+\nu-1+1-\beta-\frac{1}{q}\right)}\\
&=\frac{(1-p^{-1})}{\Gamma(\delta)}
\sum_{j=1}^{\infty}
p^{j\left(-\nu \gamma-\beta-\frac{1}{q}\right)}\\
&=\frac{(1-p^{-1})}{\Gamma(\delta)}
\sum_{j=1}^{\infty}
p^{-j\left(\nu \gamma+\beta+\frac{1}{q}\right)} < \infty.
\end{align*}
Since this is a convergent geometric series and $\nu \gamma+\beta>-\frac{1}{q}$, it follows that the constant $C_{p,\beta,q}$is finite.
Therefore, by Theorem~\ref{thm:main}, the operator $J^{\nu}_{\gamma,\delta}$ is bounded on the fractional Ces\`aro--Morrey space $C^{\beta}M_{r,q}(\mathbb{Q}_p)$.
\end{proof}

%%%%%%%%%%%%%%%%%%%%%%%%%%%%%%%%%%%%%%%%%%%%%%%%%%%%%%%%%%%%%%%%%%%%%%%%%%%%%%%%%

\section{$p$-adic Hadamard type Fractional integrals}
In this section, we introduce the $p$-adic analogues of Hadamard-type fractional integral operators, including both the left-sided and right-sided forms. These operators extend the classical Hadamard fractional integrals to the non-Archimedean setting of $\mathbb{Q}_p$ in a natural way. Our aim is to analyze their behavior within the framework of fractional Ces\`aro--Morrey spaces. In particular, as an application of Theorem~\ref{thm:main}, we establish the mapping properties of these operators on $C^{\beta}M_{r,q}(\mathbb{Q}_p)$. For the classical definitions, notations, and mapping properties on various function spaces, we refer the reader to \cite{BP2, BP1,BP3}.
Let $\alpha>0$ and $\mu \in \mathbb{R}$. Let $f$ be a measurable function on $\mathbb{Q}_p$. 
For each $s \in \mathbb{Q}_p^{\ast}$, the left-sided and right-sided $p$-adic Hadamard-type 
fractional integral operators are defined as follows.
We define the left-sided operator by
\begin{align*}
(\mathcal{J}^{\alpha}_{0+,\mu} f)(s)
=
\int_{\mathbb{Q}_p^{\ast}} 
J^{\alpha}_{0+,\mu}(|s|_p,|t|_p)\, f(t)\,dt
\end{align*}
where the kernel is given by
\begin{align}
J^{\alpha}_{0+,\mu}(|s|_p,|t|_p)
=
\frac{1}{\Gamma(\alpha)}
\chi_{\{t\in\mathbb{Q}_p:\,|t|_p<|s|_p\}}(t)
\left(\frac{|t|_p}{|s|_p}\right)^{\mu}
\left(\log \frac{|s|_p}{|t|_p}\right)^{\alpha-1}
\frac{1}{|t|_p}.
\end{align}

Similarly, the right-sided operator is defined by
\begin{align*}
(\mathcal{J}^{\alpha}_{-,\mu} f)(s)
=
\int_{\mathbb{Q}_p^{\ast}} 
J^{\alpha}_{-,\mu}(|s|_p,|t|_p)\, f(t)\,dt
\end{align*}
where
\begin{align}
J^{\alpha}_{-,\mu}(|s|_p,|t|_p)
=
\frac{1}{\Gamma(\alpha)}
\chi_{\{t\in\mathbb{Q}_p:\,|t|_p>|s|_p\}}(t)
\left(\frac{|s|_p}{|t|_p}\right)^{\mu}
\left(\log \frac{|t|_p}{|s|_p}\right)^{\alpha-1}
\frac{1}{|t|_p}.
\end{align}

The corresponding $I$-type operators are defined as follows. The left-sided case is
\begin{align*}
(\mathcal{I}^{\alpha}_{0+,\mu} f)(s)
=
\int_{\mathbb{Q}_p^{\ast}} 
I^{\alpha}_{0+,\mu}(|s|_p,|t|_p)\, f(t)\,dt
\end{align*}
with kernel
\begin{align}
I^{\alpha}_{0+,\mu}(|s|_p,|t|_p)
=
\frac{1}{\Gamma(\alpha)}
\chi_{\{t\in\mathbb{Q}_p:\, |t|_p < |s|_p\}}(t)
\left(\frac{|t|_p}{|s|_p}\right)^{\mu}
\left(\log \frac{|s|_p}{|t|_p}\right)^{\alpha-1}
\frac{1}{|s|_p},
\end{align}
and
\begin{align*}
(\mathcal{I}^{\alpha}_{-,\mu} f)(s)
=
\int_{\mathbb{Q}_p^{\ast}} 
I^{\alpha}_{-,\mu}(|s|_p,|t|_p)\, f(t)\,dt
\end{align*}
with kernel
\begin{align}
I^{\alpha}_{-,\mu}(|s|_p,|t|_p)
=
\frac{1}{\Gamma(\alpha)}
\chi_{\{t\in\mathbb{Q}_p:\, |t|_p >|s|_p\}}(t)
\left(\frac{|s|_p}{|t|_p}\right)^{\mu}
\left(\log \frac{|t|_p}{|s|_p}\right)^{\alpha-1}
\frac{1}{|s|_p}.
\end{align}
\begin{theorem}
Let Let $1\le r\le q<\infty$, $\beta \in [0,1)$, $\alpha \in (0,\infty)$, and $\mu \in \mathbb{R}$.

\begin{itemize}
\item[(a)] If  $\mu > \beta + \frac{1}{q}$, then the left-sided Hadamard-type fractional integral operator 
$\mathcal{J}^{\alpha}_{0+,\mu} f$ is bounded on $C^{\beta} M_{r,q}(\mathbb{Q}_p)$.

\item[(b)] If  $\mu > -\beta - \frac{1}{q}$, then the right-sided Hadamard-type fractional integral operator 
$\mathcal{J}^{\alpha}_{-,\mu} f$ is bounded on $C^{\beta} M_{r,q}(\mathbb{Q}_p)$.

\item[(c)] If  $\mu > \beta + \frac{1}{q} - 1$, then the left-sided Hadamard-type fractional integral operator 
$\mathcal{I}^{\alpha}_{0+,\mu} f$ is bounded on $C^{\beta} M_{r,q}(\mathbb{Q}_p)$.

\item[(d)] If  $\mu > 1 - \beta - \frac{1}{q}$, then the right-sided Hadamard-type fractional integral operator 
$\mathcal{I}^{\alpha}_{-,\mu} f$ is bounded on $C^{\beta} M_{r,q}(\mathbb{Q}_p)$.
\end{itemize}
\end{theorem}

\begin{proof}

\begin{enumerate}

\item[(a)]  For any $a \in \mathbb{Q}_p^{\ast}$, we obtain
\begin{align*}
J^{\alpha}_{p+,\mu}(|as|_p,|at|_p)
&=
\frac{1}{\Gamma(\alpha)}
\chi_{\{at\in\mathbb{Q}_p:\, |at|_p < |as|_p\}}(at)
\left(\frac{|at|_p}{|as|_p}\right)^{\mu}
\left(\log \frac{|as|_p}{|at|_p}\right)^{\alpha-1}
\frac{1}{|at|_p} \\
&=
\frac{1}{|a|_p}
\frac{1}{\Gamma(\alpha)}
\chi_{\{t\in\mathbb{Q}_p:\, |t|_p < |s|_p\}}(t)
\left(\frac{|t|_p}{|s|_p}\right)^{\mu}
\left(\log \frac{|s|_p}{|t|_p}\right)^{\alpha-1}
\frac{1}{|t|_p} \\
&=
|a|_p^{-1} J^{\alpha}_{p+,\mu}(|s|_p,|t|_p).
\end{align*}

Thus, the kernel satisfies the homogeneity condition. Therefore, it remains to show that $C_{p,\beta,q} < \infty$. Now,
\begin{align*}
C_{p,\beta,q}
&=
(1-p^{-1}) \sum_{j\in\mathbb{Z}} J^{\alpha}_{p+,\mu}(1,p^j)\, p^{j\left(1-\beta-\frac{1}{q}\right)} \\
&=
\frac{1-p^{-1}}{\Gamma(\alpha)} 
\sum_{j\in\mathbb{Z}} 
\chi_{B(0,1)}(p^j)\,(p^j)^{\mu}
\left(\log \frac{1}{p^j}\right)^{\alpha-1}
\frac{1}{p^j}
\, p^{j\left(1-\beta-\frac{1}{q}\right)} \\
&=
\frac{1-p^{-1}}{\Gamma(\alpha)} 
\sum_{j=-\infty}^{-1}
(-j\log p)^{\alpha-1}
\, p^{j\left(\mu-\beta-\frac{1}{q}\right)}.
\end{align*}

Setting $j=-k$, we obtain
\begin{align*}
C_{p,\beta,q}
&=
\frac{1-p^{-1}}{\Gamma(\alpha)} 
(\log p)^{\alpha-1}
\sum_{k=1}^{\infty}
k^{\alpha-1}
\, p^{-k\left(\mu-\beta-\frac{1}{q}\right)}.
\end{align*}

The above series is convergent by the Cauchy root test whenever 
 $\mu > \beta + \frac{1}{q}$. Hence,
\[
C_{p,\beta,q} < \infty.
\]

This completes the proof of part (a).

\item[(b)] 
Let $a \in \mathbb{Q}_p^{\ast}$. Consider the kernel
\begin{align*}
J^{\alpha}_{-,\mu}(|s|_p,|t|_p)
=
\frac{1}{\Gamma(\alpha)} \,
\chi_{\{t\in\mathbb{Q}_p:\, |t|_p > |s|_p\}}(t)
\left(\frac{|s|_p}{|t|_p}\right)^{\mu}
\left(\log \frac{|t|_p}{|s|_p}\right)^{\alpha-1}
\frac{1}{|t|_p}.
\end{align*}

Then we have
\begin{align*}
J^{\alpha}_{-,\mu}(|as|_p,|at|_p)
=
\frac{1}{\Gamma(\alpha)} \,
\chi_{\{at\in\mathbb{Q}_p:\, |at|_p > |as|_p\}}(at)
\left(\frac{|as|_p}{|at|_p}\right)^{\mu}
\left(\log \frac{|at|_p}{|as|_p}\right)^{\alpha-1}
\frac{1}{|at|_p}.
\end{align*}

Using the multiplicative property of the $p$-adic norm, this becomes
\begin{align*}
J^{\alpha}_{-,\mu}(|as|_p,|at|_p)
&=
\frac{1}{|a|_p} \,
\frac{1}{\Gamma(\alpha)} \,
\chi_{\{t\in\mathbb{Q}_p:\, |t|_p > |s|_p\}}(t)
\left(\frac{|s|_p}{|t|_p}\right)^{\mu}
\left(\log \frac{|t|_p}{|s|_p}\right)^{\alpha-1}
\frac{1}{|t|_p}\\
&=
|a|_p^{-1} J^{\alpha}_{-,\mu}(|s|_p,|t|_p).
\end{align*}

Hence, the kernel satisfies the homogeneity condition.

Now,
\begin{align*}
C_{p,\beta,q}
&=
(1-p^{-1}) \sum_{j\in\mathbb{Z}} J^{\alpha}_{-,\mu}(1,p^j)\, p^{j\left(1-\beta-\frac{1}{q}\right)} \\
&=
\frac{1-p^{-1}}{\Gamma(\alpha)}
\sum_{j>0}
\left(\frac{1}{p^j}\right)^{\mu}
\left(\log p^j\right)^{\alpha-1}
\frac{1}{p^j}
\, p^{j\left(1-\beta-\frac{1}{q}\right)} \\
&=
\frac{1-p^{-1}}{\Gamma(\alpha)}
\sum_{j=1}^{\infty}
p^{-j\mu} (j\log p)^{\alpha-1} p^{j\left(1-\beta-\frac{1}{q}-1\right)} \\
&=
\frac{1-p^{-1}}{\Gamma(\alpha)} (\log p)^{\alpha-1}
\sum_{j=1}^{\infty}
j^{\alpha-1} p^{-j\left(\mu+\beta+\frac{1}{q}\right)}.
\end{align*}

The above series converges by the Cauchy root test whenever $\mu > -\beta - \frac{1}{q}$. Hence,
\begin{align*}
C_{p,\beta,q} < \infty.
\end{align*}

This completes the proof of part (b).

\item[(c)]

Let $a \in \mathbb{Q}_p^{\ast}$. Then
\begin{align*}
I^{\alpha}_{0+,\mu}(|s|_p,|t|_p)
&=
\frac{1}{\Gamma(\alpha)} \,
\chi_{\{t\in\mathbb{Q}_p:\, |t|_p < |s|_p\}}(t)
\left(\frac{|t|_p}{|s|_p}\right)^{\mu}
\left(\log \frac{|s|_p}{|t|_p}\right)^{\alpha-1}
\frac{1}{|s|_p}.
\end{align*}

Now,
\begin{align*}
I^{\alpha}_{0+,\mu}(|as|_p,|at|_p)
&=
\frac{1}{\Gamma(\alpha)} \,
\chi_{\{at\in\mathbb{Q}_p:\, |at|_p < |as|_p\}}(at)
\left(\frac{|at|_p}{|as|_p}\right)^{\mu}
\left(\log \frac{|as|_p}{|at|_p}\right)^{\alpha-1}
\frac{1}{|as|_p}.
\end{align*}

Using the multiplicative property of the $p$-adic norm, we obtain
\begin{align*}
I^{\alpha}_{0+,\mu}(|as|_p,|at|_p)
&=
\frac{1}{|a|_p}
\frac{1}{\Gamma(\alpha)} \,
\chi_{\{t\in\mathbb{Q}_p:\, |t|_p < |s|_p\}}(t)
\left(\frac{|t|_p}{|s|_p}\right)^{\mu}
\left(\log \frac{|s|_p}{|t|_p}\right)^{\alpha-1}
\frac{1}{|s|_p} \\
&=
|a|_p^{-1} I^{\alpha}_{0+,\mu}(|s|_p,|t|_p).
\end{align*}
Hence, the kernel $I^{\alpha}_{0+,\mu}(|s|_p,|t|_p)$ satisfies the homogeneity condition.

Furthermore,
\begin{align*}
C_{p,\beta,q}
&=
(1-p^{-1}) \sum_{j\in\mathbb{Z}} I^{\alpha}_{0+,\mu}(1,p^j)\, p^{j\left(1-\beta-\frac{1}{q}\right)} \\
&=
\frac{1-p^{-1}}{\Gamma(\alpha)}
\sum_{j\in\mathbb{Z}}
\chi_{B(0,1)}(p^j)\,(p^j)^{\mu}
\left(\log \frac{1}{p^j}\right)^{\alpha-1}
\, p^{j\left(1-\beta-\frac{1}{q}\right)} \\
&=
\frac{1-p^{-1}}{\Gamma(\alpha)}
\sum_{j<0}
p^{j\mu} (-j\log p)^{\alpha-1}
\, p^{j\left(1-\beta-\frac{1}{q}\right)} \\
&=
\frac{1-p^{-1}}{\Gamma(\alpha)} (\log p)^{\alpha-1}
\sum_{j=-\infty}^{-1}
(-j)^{\alpha-1}
\, p^{j\left(1+\mu-\beta-\frac{1}{q}\right)} \\
&=
\frac{1-p^{-1}}{\Gamma(\alpha)} (\log p)^{\alpha-1}
\sum_{j=1}^{\infty}
j^{\alpha-1}
\, p^{-j\left(1+\mu-\beta-\frac{1}{q}\right)}.
\end{align*}

The above series is convergent. Since $\mu > \beta + \frac{1}{q} - 1$, by the Cauchy root test we obtain the convergence of the series, and hence
\begin{align*}
C_{p,\beta,q} < \infty.
\end{align*}

\item[(d)]

In a similar way, it can be easily seen that the kernel $I^{\alpha}_{-,\mu}(|s|_p,|t|_p)$ satisfies the homogeneity condition. 
Therefore, it remains to show only the finiteness of the constant $C_{p,\beta,q}$.

We have
\begin{align*}
C_{p,\beta,q}
&=
(1-p^{-1}) \sum_{j\in\mathbb{Z}} I^{\alpha}_{-,\mu}(1,p^j)\, p^{j\left(1-\beta-\frac{1}{q}\right)} \\
&=
\frac{1-p^{-1}}{\Gamma(\alpha)} 
\sum_{j>0}
\left(\frac{1}{p^j}\right)^{\mu}
\left(\log (p^j)\right)^{\alpha-1}
\, p^{j\left(1-\beta-\frac{1}{q}\right)} \\
&=
\frac{1-p^{-1}}{\Gamma(\alpha)} 
\sum_{j=1}^{\infty}
p^{-j\mu} (\log p)^{\alpha-1} j^{\alpha-1} p^{j\left(1-\beta-\frac{1}{q}\right)} \\
&=
\frac{1-p^{-1}}{\Gamma(\alpha)} (\log p)^{\alpha-1}
\sum_{j=1}^{\infty}
j^{\alpha-1} p^{j\left(1-\beta-\frac{1}{q}-\mu\right)} \\
&=
\frac{1-p^{-1}}{\Gamma(\alpha)} (\log p)^{\alpha-1}
\sum_{j=1}^{\infty}
j^{\alpha-1} p^{-j\left(\mu+\beta+\frac{1}{q}-1\right)}.
\end{align*}

The given series is convergent by the Cauchy root test. Since 
\[
\mu > 1 - \beta - \frac{1}{q},
\]
we obtain the convergence of the series, and hence
\begin{align*}
C_{p,\beta,q} < \infty.
\end{align*}

This completes the proof of part (d).

\end{enumerate}

\end{proof}
\section*{Ethics declarations}
\subsection*{Conflict of interest}  The authors declare no conflict of interests.

\bibliographystyle{abbrv}
\bibliography{references}

\end{document}